\documentclass[11pt,headings=small,BCOR=15mm]{scrartcl}

\usepackage[sumlimits]{amsmath}
\usepackage{amssymb,amsfonts,amsthm}
\usepackage[english]{babel}
\usepackage[latin1]{inputenc}
\usepackage[T1]{fontenc}
\usepackage{lmodern}
\usepackage{graphicx}
\usepackage{scrpage2}
\usepackage{hyperref}
\usepackage{setspace}
\usepackage{textcomp}
\usepackage{enumerate}
\usepackage{xfrac}
\usepackage{helvet}
\usepackage{bbm}

\newtheorem{thm}{Theorem}[section]
\newtheorem{prp}[thm]{Proposition}
\newtheorem{lem}[thm]{Lemma}

\newtheorem*{rem}{Remark}

\def\R{\mathbb{R}}
\def\N{\mathbb{N}}
\def\Q{I}

\def\D{\mathbb{D}}
\def\B{\mathbb{B}}
\def\1{\mathbbm{1}}
\def\Rn{\mathcal{R}_n}

\newcommand\dint{{\rm d}}

\DeclareMathOperator{\Err}{Err}
\DeclareMathOperator{\supp}{supp}

\allowdisplaybreaks[1]
\clearscrheadfoot
\ihead{\headmark}
\ohead{\thepage}

\begin{document}
\pagestyle{scrheadings}
\onehalfspacing

\newlength{\fixboxwidth}
\setlength{\fixboxwidth}{\marginparwidth}
\addtolength{\fixboxwidth}{-7pt}
\newcommand{\fix}[1]{\marginpar{\fbox{\parbox{\fixboxwidth}{\raggedright\tiny #1}}}}

\title{Discrepancy of generalized Hammersley type point sets in Besov spaces with dominating mixed smoothness}
\author{Lev Markhasin\thanks{Research of the author was supported by a scholarship of Ernst Ludwig Ehrlich Studienwerk.}\\
        \tiny{Friedrich-Schiller-Universit\"at Jena, e-mail: lev.markhasin@uni-jena.de}}
\date{\today}
\maketitle

\begin{abstract}
The symmetrized Hammersley point set is known to achieve the best possible rate for the $L_2$-norm of the discrepancy function. Also lower bounds for the norm in Besov spaces with dominating mixed smoothness are known. In this paper a large class of point sets which are generalizations of the Hammersley type point sets are proved to asymptotically achieve the known lower bound of the Besov norm. The proof uses a $b$-adic generalization of the Haar system. This result can be regarded as a preparation for the proof in arbitrary dimension.
\end{abstract}

\noindent{\footnotesize {\it 2010 Mathematics Subject Classification.} Primary 11K06,11K38,42C10,46E35,65C05.}\\[5mm]
\textit{Acknowledgement:} The author wants to thank Aicke Hinrichs and Hans Triebel for useful discussions and an anonymous referee for suggestions to improve the presentation.

\section{Introduction}
Let $N$ be some positive integer and $\mathcal{P}$ a point set in the unit cube $\Q^d = [0,1)^d$ with $N$ points. Then the discrepancy function $D_{\mathcal{P}}$ is defined as
\begin{align}
D_{\mathcal{P}}(x) = \frac{1}{N} \sum_{z \in \mathcal{P}} \1_{C_z}(x) - |B_x|.
\end{align}
By $|B_x| = x_1 \cdot \ldots \cdot x_d$ we denote the volume of the rectangular box $B_x = [0,x_1) \times \ldots \times [0,x_d)$ where $x = (x_1,\ldots,x_d) \in \Q^d$ while $\1_{C_z}$ is the characteristic function of the rectangular box $C_z = (z_1,1) \times \ldots \times (z_d,1)$ for $z \in \mathcal{P}$.

Usually one is interested in calculating the norm of the discrepancy function in some normed space of functions on $\Q^d$ which contain the discrepancy function. A very well known result refers to the space $L_2(\Q^d)$. It was proved by Roth in \cite{R54}. There exists a constant $c_1 > 0$ such that for any $N \geq 1$ the discrepancy function of any point set $\mathcal{P}$ in $\Q^d$ with $N$ points satisfies
\[ \left\| D_{\mathcal{P}} | L_2 \right\| \geq c_1 \, \frac{(\log N)^\frac{d-1}{2}}{N}. \]

The currently best known values for the constant $c_1$ can be found in \cite{HM11}. Furthermore, there exists a constant $c_2 > 0$ such that for any $N \geq 1$, there exists a point set $\mathcal{P}$ in $\Q^d$ with $N$ points that satisfies
\[ \left\| D_{\mathcal{P}} | L_2 \right\| \leq c_2 \, \frac{(\log N)^\frac{d-1}{2}}{N}. \]
This result is known for dimension $2$ from \cite{D56} (Davenport), for dimension $3$ from \cite{R79} (Roth) and for arbitrary dimension from \cite{R80} (Roth). Only Davenport's result has been proved by an explicit construction while for higher dimensions probabilistic methods were used until Chen and Skriganov found explicit constructions for arbitrary dimension in \cite{CS02}. Results for the constant $c_2$ can be found in \cite{FPPS10}.

Both bounds were extended to $L_p$-spaces for any $1 < p < \infty$. In the case of the lower bound the reference is \cite{S77} (Schmidt) while for the upper bound it is \cite{C80} (Chen).

As general references for studies of the discrepancy function we refer to the recent monographs \cite{DP10} and \cite{NW10} as well as \cite{M99}, \cite{KN74} and \cite{B11}.

Until recently other norms than $L_p$-norms weren't studied a lot in the context of discrepancy. Triebel started the study of the discrepancy function in other function spaces like Sobolev, Besov and Triebel-Lizorkin spaces in \cite{T10b} and \cite{T10a}. In \cite{H10} Hinrichs proved sharp upper bounds for the norms in Besov spaces with dominating mixed smoothness. Triebel's result was that for all $1 \leq p,q \leq \infty$ and $r \in \R$ satisfying $\frac{1}{p} - 1 < r < \frac{1}{p}$ and $q < \infty$ if $p = 1$ and $q > 1$ if $p = \infty$ there exist constants $c_1, c_2 > 0$ such that, for any $N \geq 2$, the discrepancy function of any point set $\mathcal{P}$ in $\Q^d$ with $N$ points satisfies
\begin{align} \label{triebel_res}
\left\| D_{\mathcal{P}} | S_{pq}^r B(\Q^d) \right\| \geq c_1 \, N^{r-1} (\log N)^{\frac{d-1}{q}},
\end{align}
and, for any $N \geq 2$, there exists a point set  $\mathcal{P}$ in $\Q^d$ with $N$ points such that
\[ \left\| D_{\mathcal{P}} | S_{pq}^r B(\Q^d) \right\| \leq c_2 \, N^{r-1} (\log N)^{(d-1)(\frac{1}{q} + 1 - r)}. \]
Hinrichs' result closed this gap in the case $d = 2$, we will mention it later.

This note will closely orient itself on \cite{H10} in terms of structure and methods of proofs. We mention some definitions from \cite{T10a} which are most important for our purpose.

Let $\mathcal{S}(\R^d)$ denote the Schwartz space and $\mathcal{S}'(\R^d)$ the space of tempered distributions on $\R^d$. For $f \in \mathcal{S}'(\R^d)$, we denote by $\mathcal{F}f$ the Fourier transform of $f$. Let $\varphi_0 \in \mathcal{S}(\R^d)$ satisfy $\varphi_0(t) = 1$ for $|t| \leq 1$ and $\varphi_0(t) = 0$ for $|t| > \frac{3}{2}$. Let
\[ \varphi_k(t) = \varphi_0(2^{-k} t) - \varphi_0(2^{-k + 1} t) \]
where $k \in \N, \,  t \in \R$ and
\[ \varphi_k(t) = \varphi_{k_1}(t_1) \ldots \varphi_{k_d}(t_d) \]
where $k = (k_1,\ldots,k_d) \in \N_0^d, \, t = (t_1,\ldots,t_d) \in \R^d$.
The functions $\varphi_k$ are a dyadic resolution of unity since
\[ \sum_{k \in \N_0^d} \varphi_k(x) = 1 \]
for all $x \in \R^d$. The functions $\mathcal{F}^{-1}(\varphi_k \mathcal{F} f)$ are entire analytic functions for any $f \in \mathcal{S}'(\R^d)$.

Let $0 < p,q \leq \infty$ and $r \in \R$. The Besov space with dominating mixed smoothness $S_{pq}^r B(\R^d)$ consists of all $f \in \mathcal{S}'(\R^d)$ with finite quasi-norm
\[ \left\| f | S_{pq}^r B(\R^d) \right\| = \left( \sum_{k \in \N_0^d} 2^{r (k_1 + \ldots + k_d) q} \left\| \mathcal{F}^{-1}(\varphi_k \mathcal{F} f) | L_p(\R^d) \right\|^q \right)^{\frac{1}{q}} \]
with the usual modification if $q = \infty$.

Let $\mathcal{D}(\Q^d)$ consist of all complex-valued infinitely differentiable functions on $\R^d$ with compact support in the interior of $\Q^d$ and let $\mathcal{D}'(\Q^d)$ be its dual space of all distributions in $\Q^d$. The Besov space with dominating mixed smoothness $S_{pq}^r B(\Q^d)$ consists of all $f \in \mathcal{D}'(\Q^d)$ with finite quasi-norm
\[ \left\| f | S_{pq}^r B(\Q^d) \right\| = \inf \left\{ \left\| g | S_{pq}^r B(\R^d) \right\| : \: g \in S_{pq}^r B(\R^d), \: g|_{\Q^d} = f \right\}. \]
The spaces $S_{pq}^r B(\R^d)$ and $S_{pq}^r B(\Q^d)$ are quasi-Banach spaces.

In \cite{H10} Hinrichs analyzed the norm of the discrepancy function of point sets of the Hammersley type in Besov spaces with dominating mixed smoothness. He proved upper bounds which are special cases of our results in this note. The result from \cite{H10} is that for $r \geq 0$ there is a constant $c > 0$ such that for any $N \geq 2$, there exists a point set $\mathcal{P}$ in $\Q^2$ with $N$ points such that
\[ \left\| D_{\mathcal{P}} | S_{pq}^r B(\Q^2) \right\| \leq c \, N^{r-1} (\log N)^{\frac{1}{q}}. \]
This result closed the gap of Triebel's results in dimension $2$. In this note we prove the same bound for a larger class of point sets. Hinrichs used point sets of Hammersley type. We use generalizations of these point sets.

For any integer $b \geq 2$ and any $n \in \N$ we consider the following mappings
\[ s_1, \ldots, s_n: \, \{0,1,\ldots,b-1\} \longrightarrow \{0,1,\ldots,b-1\} \]
which are either defined as $s_i(t) = t$ for all $t \in \{0,1,\ldots,b-1\}$ or as $s_i(t) = b - 1 - t$ for all $t \in \{0,1,\ldots,b-1\}$ for any $1 \leq i \leq n$. Then we consider point sets
\begin{multline*}
\Rn = \Big\{ \Big( \frac{t_n}{b} + \frac{t_{n-1}}{b^2} + \ldots + \frac{t_1}{b^n},\frac{s_1(t_1)}{b} + \frac{s_{2}(t_2)}{b^2} + \ldots + \frac{s_n(t_n)}{b^n} \Big): \\
t_1,\ldots,t_n\in\{0,1,\ldots,b-1\} \Big\}.
\end{multline*}
So, the set $\Rn$ contains exactly $b^n$ points. These sets are called generalized Hammersley type point sets since they generalize original Hammersley type point sets proposed by Hammersley in \cite{H60}. They were first defined by Faure in \cite{F81} and used in \cite{FP09} and \cite{FPPS10} to calculate their $L_2$-discrepancy. We abbreviate $s_i = s_i(t_i)$ for all $i$.

The explicit constructions for the $L_2$-discrepancy by Chen and Skriganov from \cite{CS02} use $b$-adic constructions, similar to the $b$-adic generalizations of the Hammersley type point sets for $d \geq 2$. One might conjecture that these constructions could be optimal for the norms in Besov spaces with dominating mixed smoothness for arbitrary dimension. Considering this aspect, one could see the current paper as the preparation for the proof of this conjecture.

For any point set $\Rn$ we denote $a_n = \#\{i = 1,\ldots,n:\,s_i = t_i\}$. The main result of this note is

\begin{thm} \label{thm_hammersley_disc}

Let $1 \leq p,q \leq \infty$ and $0 \leq r < \frac{1}{p}$. Then for any integer $b\geq 2$ there is a constant $c > 0$ such that for any $n \in \N$ and any generalized Hammersley type point set $\Rn$ with $a_n$ satisfying $|2a_n - n| \leq c_0$ for some constant $c_0 > 0$, we have
\[ \left\| D_{\Rn} | S_{pq}^r B(\Q^2) \right\| \leq c \, b^{n(r-1)} \, n^{\frac{1}{q}}. \]

\end{thm}

\begin{rem}

The constant $c_0$ is independent of $n$, securing that $|2a_n - n|$ can be estimated with the same constant for any $n$ and any possible $\Rn$. In \cite{H10} only point sets with $a_n = \left\lfloor  \frac{n}{2} \right\rfloor$ were used (with $b = 2$). So a possible value for $c_0$ in that case would be $1$.

\end{rem}

In order to prove the result we will calculate $b$-adic Haar coefficients of the discrepancy function.

The distribution of points in a cube is not just a theoretical concept. Its application in quasi-Monte Carlo methods is very important. Quadrature formulas need very well distributed point sets. The connection of discrepancy and the error of quadrature formulas can be given for a lot of norms. In \cite[Theorem 6.11]{T10a} Triebel gave this connection for Besov spaces with dominating mixed smoothness. We define the error of the quadrature formulas in some Banach space $M(\Q^d)$ of functions on $\Q^d$ with $N$ points as
\[ \Err_N(M(\Q^d)) = \inf_{\{ x_1, \ldots, x_N \} \subset \Q^d} \sup_{f \in M^1_0(\Q^d)} \left| \int_{\Q^d} f(x) \dint x - \frac{1}{N}\sum_{k = 1}^N f(x_k) \right| \]
where by $M^1_0(\Q^d)$ we mean the subset of the unit ball of $M(\Q^d)$ with the property that for all $f \in M^1_0(\Q^d)$ its extension to $\overline{\Q^d}$ vanishes whenever one of the coordinates of the argument is $1$.

\begin{thm}

Let $1 \leq p,q \leq \infty$ and $\frac{1}{p} < r \leq 1$. Then for any integer $b\geq 2$ there are constants $c_1, c_2 > 0$ such that, for any $n \in \N$ and any generalized Hammersley type point set $\Rn$ with $a_n$ satisfying $|2a_n - n| \leq c_0$ for some constant $c_0 > 0$, we have
\item \[ c_1 \, \frac{(\log N)^{\frac{(q - 1)(d-1)}{q}}}{N^r} \leq \Err_N(S_{pq}^r B(\Q^d)) \leq c_2 \, \frac{(\log N)^{\frac{(q - 1)(d-1)}{q}}}{N^r}, \]

\end{thm}

\begin{proof}
This follows from \eqref{triebel_res} and Theorem \ref{thm_hammersley_disc} in combination with \cite[Theorem 6.11]{T10a}.

\end{proof}

\section{The $b$-adic Haar bases}

For some integer $b\geq 2$ a $b$-adic interval of length $b^{-j},\, j\in\N_0$ in $\Q$ is an interval of the form
\[ I_{jm} = I_{jm}^b = \big[ b^{-j} m, b^{-j} (m+1) \big)\]
for $m=0,1,\ldots,b^j-1$. For $j \in \N_0$ we divide $I_{jm}$ into $b$ intervals of length $b^{-j - 1}$, i.e. $I_{jm}^k = I_{jm}^{b,k} = I_{j + 1,bm + k},\, k=0,\ldots,b - 1$. As an additional notation we put $I_{-1,0}^{-1} = I_{-1,0} = [0,1)$. Let $\D_j = \{0,1,\ldots,b^j-1\}$ and $\B_j = \{1,\ldots,b-1\}$ for $j \in \N_0$ and $\D_{-1} = \{0\}$ and $\B_{-1} = \{1\}$. The $b$-adic Haar functions $h_{jm\ell} = h_{jm\ell}^b,$ have support in $I_{jm}$. For any $j \in \N_0,\, m \in \D_j,\, \ell \in \B_j$ and any $k=0,\ldots,b-1$ the value of $h_{jm\ell}$ in $I_{jm}^k$ is $e^{\frac{2\pi i}{b}\ell k}$. We denote the indicator function of $I_{-1,0}$ by $h_{-1,0,1}$. Let $\N_{-1} = \{-1,0,1,2,\ldots\}$. The functions $h_{jm\ell},\, j \in \N_{-1},\,m \in \D_j,\,\ell\in\B_j$ are called $b$-adic Haar system. Normalized in $L_2(\Q)$ we obtain the orthonormal $b$-adic Haar basis of $L_2(\Q)$. The proof of this fact can be found in \cite{RW98}.

For $j = (j_1, \dots, j_d) \in \N_{-1}^d$, $m = (m_1, \ldots, m_d) \in \D_j := \D_{j_1} \times \ldots \times \D_{j_d}$ and $\ell = (\ell_1, \ldots, \ell_d) \in \B_j := \B_{j_1} \times \ldots \times \B_{j_d}$, the Haar function $h_{jm\ell}$ is given as the tensor product $h_{jm\ell}(x) = h_{j_1,m_1,\ell_1}(x_1) \ldots h_{j_d,m_d,\ell_d}(x_d)$ for $x = (x_1, \ldots, x_d) \in \Q^d$. We will call $I_{jm} = I_{j_1,m_1} \times \ldots \times I_{j_d,m_d}$ $b$-adic boxes. For $k = (k_1,\ldots,k_d)$ where $k_i \in \{ 0,\ldots,b - 1 \}$ for $j_i \in \N_0$ and $k_i = -1$ for $j_i = -1$ we put $I_{jm}^k = I_{j_1 m_1}^{k_1} \times \ldots \times I_{j_d m_d}^{k_d}$. The functions $h_{jm\ell},\, j \in \N_{-1}^d,\,m \in \D_j,\,\ell\in\B_j$ are called $d$-dimensional $b$-adic Haar system. Normalized in $L_2(\Q^d)$ we obtain the orthonormal $b$-adic Haar basis of $L_2(\Q^d)$.

For any function $f \in L_2(\Q^d)$ we have by Parseval's equation
\begin{align}
\|f|L_2\|^2 = \sum_{j\in\N_{-1}^d} b^{\max(0,j_1) + \ldots + \max(0,j_d)} \sum_{m\in \D_j, \ell\in\B_j} |\mu_{jm\ell}|^2.
\end{align}
where
\begin{align}
\mu_{jm\ell} = \mu_{jm\ell}(f) = \int_{\Q^d} f(x) h_{jm\ell}(x) \, \dint x
\end{align}
are the $b$-adic Haar coefficients of $f$.

Our goal is to combine the $b$-adic Haar basis method with Triebel's theory in Besov spaces. We generalize \cite[Theorem 2.41]{T10a} for $b$-adic Haar systems in the $d$-dimensional unit cube. So we characterize Besov spaces $S_{pq}^rB(\Q^d)$ with dominating mixed smoothness. 

\section{Characterization for Besov spaces with dominating mixed smoothness}

\begin{thm} \label{thm_besov_char}

Let $0 < p,q \leq \infty$ and $\frac{1}{p} - 1 < r < \min(\frac{1}{p},1)$. Let $f \in D'(\Q^d)$. Then $f\in S_{pq}^r B(\Q^d)$ if and only if it can be represented as
\begin{align} \label{f_represented_h}
f = \sum_{j \in \N_{-1}^d} \sum_{m \in \D_j, \, \ell \in \B_j} \mu_{jm\ell} \, b^{\max(0,j_1) + \ldots + \max(0,j_d)} h_{jm\ell}
\end{align}
for some sequence $(\mu_{jm\ell})$ satisfying
\begin{align} \label{eq_quasinorm}
\left( \sum_{j\in\N_{-1}^d} b^{(j_1 + \ldots + j_d)(r - \frac{1}{p} + 1) q} \left( \sum_{m \in \D_j, \, \ell \in \B_j} | \mu_{jm\ell}|^p \right)^{\frac{q}{p}} \right)^{\frac{1}{q}} < \infty,
\end{align}
where the convergence is unconditional in $D'(\Q^d)$ and in any $S_{pq}^\rho B(\Q^d)$ with $\rho<r$. This representation of $f$ is unique with the $b$-adic Haar coefficients
\[ \mu_{jm\ell} = \mu_{jm\ell}^b(f) = \int_{\Q^d}f(x)h_{jm\ell}(x) dx. \]
The expression \eqref{eq_quasinorm} additionally delivers an equivalent quasi-norm on $S_{pq}^r B(\Q^d)$.

\end{thm}

The definition of the spaces $S_{pq}^r B (\Q^d)$ was dyadic therefore, making it difficult to gain any $b$-adic results. Hence, we have to change the base first.

Let $\varphi_0 \in \mathcal{S}(\R)$ satisfy $\varphi_0(t) = 1$ for $|t| \leq 1$ and $\varphi_0(t) = 0$ for $|t| > \frac{b + 1}{b}$. Let
\[ \varphi_k(t) = \varphi_0(b^{-k} t) - \varphi_0(b^{-k + 1} t) \]
where $t \in \R, \, k \in \N$ and
\[ \varphi_k(t) = \varphi_{k_1}(t_1) \ldots \varphi_{k_d}(t_d) \]
where $k = (k_1,\ldots,k_d) \in \N_0^d, \, t = (t_1,\ldots,t_d) \in \R^d$.
The functions $\varphi_k$ are a $b$-adic resolution of unity since
\[ \sum_{k \in \N_0^d} \varphi_k(x) = 1 \]
for all $x \in \R^d$. The functions $\mathcal{F}^{-1}(\varphi_k \mathcal{F} f)$ are entire analytic functions for any $f \in \mathcal{S}'(\R^d)$.
Let $0 < p,q \leq \infty$ and $r \in \R$. The $b$-adic Besov space with dominating mixed smoothness $S_{pq}^r B^b(\R^d)$ consists of all $f \in \mathcal{S}'(\R^d)$ with finite quasi-norm
\[ \left\| f | S_{pq}^r B^b(\R^d) \right\| = \left( \sum_{k \in \N_0^d} b^{r (k_1 + \ldots + k_d) q} \left\| \mathcal{F}^{-1}(\varphi_k \mathcal{F} f) | L_p(\R^d) \right\|^q \right)^{\frac{1}{q}} \]
with the usual modification if $q = \infty$. We will first prove that the $b$-adic norm is equivalent to the dyadic norm. Then we will be able to apply Triebel's ideas for the proof of the theorem. To prove the equivalence, we prove the equivalence of the $b$-adic and the $(b+1)$-adic norms. Let the functions $\varphi_k$ be a $b$-adic resolution of unity and the functions $\psi_k$ a $(b+1)$-adic resolution of unity. We observe that 
\[ \supp \varphi_k \subset [-b^{k+1}, -b^{k-1}] \cup [b^{k-1}, b^{k+1}] \]
and
\[ \supp \psi_k \subset [-(b+1)^{k+1}, -(b+1)^{k-1}] \cup [(b+1)^{k-1}, (b+1)^{k+1}]. \]
Now we check that for every $j \in \N_0$ there are at most $2$ such $k \in \N_0$ that $[b^{k-1}, b^{k+1}] \subset [(b+1)^{j-1}, (b+1)^{j+1}]$. But this is easy since $(b+1)^{j-1} \leq b^{k-1}$ and $b^{k+1} \leq (b+1)^{j+1}$ is equivalent to
\begin{align} \label{k_j_limited}
(j - 1) \frac{\log(b + 1)}{\log(b)} + 1 \leq k \leq (j + 1) \frac{\log(b + 1)}{\log(b)} - 1.
\end{align}
The fact that the cardinality of the set of such $k$ is at most $2$ follows from
\[ 2 \frac{\log(b + 1)}{\log(b)} - 2 < 2 \]
which is equivalent to
\[ \frac{\log(b + 1)}{\log(b)} < 2 \]
which is equivalent to $0 < b^2 - b - 1$ which is clearly satisfied since $b \geq 2$. Therefore, we know that for every $j$ there are not more than two $k$ such that, $\supp \varphi_k \subset \supp \psi_j$. For every $j \in \N_0$ we denote by $\Lambda(j)$ the set of such $k$ that $\supp \varphi_k \cap \supp \psi_j \neq \emptyset$. The cardinality of such sets is at most $6$ and for sure they are not empty. Conversely, for every $k \in N_0$ there are at most $3$ such $j \in \N_{-1}$ that $\supp \varphi_k \cap \supp \psi_j \neq \emptyset$. We denote by $\Omega(k)$ the set of such $j$. Additionally, we put for $j \in N_0^d$
\[ \Lambda(j) = \Lambda(j_1) \times \ldots \times \Lambda(j_d) \]
and for $k \in \N_{-1}^d$
\[ \Omega(k) = \Omega(k_1) \times \ldots \times \Omega(k_d). \]
Hence, for all $x \in \R^d$ we have
\[ \varphi_k(x) = \varphi_k(x) \sum_{j \in \Omega(k)} \psi_j(x) \]
and
\[ \psi_j(x) = \psi_j(x) \sum_{k \in \Lambda(j)} \varphi_k(x). \]
Now let $j,k \in \N_0^d$ then we have
\begin{align*}
\mathcal{F}^{-1}(\varphi_k \mathcal{F} f) = \sum_{j \in \Omega(k)} \mathcal{F}^{-1} \left( \varphi_k \mathcal{F} \left( \mathcal{F}^{-1} (\psi_{j} \mathcal{F} f) \right) \right)
\end{align*}
and
\begin{align*}
\mathcal{F}^{-1}(\psi_j \mathcal{F} f) = \sum_{k \in \Lambda(j)} \mathcal{F}^{-1} \left( \psi_j \mathcal{F} \left( \mathcal{F}^{-1} (\varphi_{k} \mathcal{F} f) \right) \right).
\end{align*}
Let $l > \frac{1}{\min(1,p)} - \frac{1}{2}$. From Lemma \cite[Proposition 2.3.3]{Hn10} for $M = \varphi_k$ and $\beta_1 = b^{k_1 + 2}, \ldots, \beta_d = b^{k_d + 2}$ we get (with a constant $c > 0$) that
\begin{align*}
& \left\| \mathcal{F}^{-1} \left( \varphi_k \mathcal{F} \left( \mathcal{F}^{-1} (\psi_{j} \mathcal{F} f) \right) \right) | L_p (\R^d) \right\| \\
& \qquad \qquad \leq c \left\| \varphi_k(b^{k_1 + 2} \cdot, \ldots, b^{k_d + 2} \cdot) | S_2^l W(\R^d) \right\| \left\| \mathcal{F}^{-1} (\psi_{j} \mathcal{F} f) | L_p(\R^d) \right\| \\
& \qquad \qquad \leq c_1 \prod_{i=1}^d \left\| \varphi_{k_i}(b^{k_i + 2} \cdot) | W_2^l(\R) \right\| \left\| \mathcal{F}^{-1} (\psi_{j} \mathcal{F} f) | L_p(\R^d) \right\|.
\end{align*}
Since $\varphi_{k_i} \in \mathcal{S}(\R)$ there exists a constant $c_2 > 0$ such that, for all $i$ we have
\[ \left\| \varphi_{k_i}(b^{k_i + 2} \cdot) | W_2^l(\R) \right\| \leq c_2. \]
Consequently, we get
\[ \left\| \mathcal{F}^{-1} \left( \varphi_k \mathcal{F} \left( \mathcal{F}^{-1} (\psi_{j} \mathcal{F} f) \right) \right) | L_p (\R^d) \right\| \leq c_3 \left\| \mathcal{F}^{-1} (\psi_{j} \mathcal{F} f) | L_p(\R^d) \right\| \]
for $j \in \Omega(k)$
and analogously (using \cite[Proposition 2.3.3]{Hn10} with $b_1 = (b + 1)^{j_1 + 2}, \ldots, b_d = (b + 1)^{j_d + 2}$)
\[ \left\| \mathcal{F}^{-1} \left( \psi_j \mathcal{F} \left( \mathcal{F}^{-1} (\varphi_{k} \mathcal{F} f) \right) \right) | L_p (\R^d) \right\| \leq c_4 \left\| \mathcal{F}^{-1} (\varphi_{k} \mathcal{F} f) | L_p(\R^d) \right\| \]
for $k \in \Lambda(j)$. So we have proved for every $k \in \N_0^d$ that
\[ \left\| \mathcal{F}^{-1} \left( \varphi_k \mathcal{F} f \right) | L_p (\R^d) \right\| \leq c \sum_{j \in \Omega(k)} \left\| \mathcal{F}^{-1} (\psi_{j} \mathcal{F} f) | L_p(\R^d) \right\|. \]
Multiplying with $b^{r (k_1 + \ldots + k_d) q}$ and summing over $k$ will give us on the left side $\left\| \cdot | S_{pq}^r B^b(\R^d) \right\|$. On the right side we get at most $3$ identical summands which we can incorporate into the constant. The norming factor can be easily estimated with a constant since the difference of $j$ and $k$ is limited by \eqref{k_j_limited}. Conversely, we have for every $j \in \N_0^d$
\[ \left\| \mathcal{F}^{-1} \left( \psi_j \mathcal{F} f \right) | L_p (\R^d) \right\| \leq c \sum_{k \in \Lambda(j)} \left\| \mathcal{F}^{-1} (\varphi_{k} \mathcal{F} f) | L_p(\R^d) \right\|. \]
Multiplying with $(b+1)^{r (j_1 + \ldots + j_d) q}$ and summing over $j$ will give us on the left side $\left\| \cdot | S_{pq}^r B^{b+1}(\R^d) \right\|$. On the right side we get at most $6$ identical summands which we can incorporate into the constant. The same applies again to the norming factor.

Now we can prove the theorem following closely the original proof from \cite{T10a}. First, one assumes for $\max(\frac{1}{p},1) - 1 < r < \min(\frac{1}{p},1)$ that the function $f$ is given in the form
\begin{align} \label{f_given_with_chi}
f = \sum_{j \in \N_0^d} b^{j_1 + \ldots + j_d} \sum_{m \in \D_j} \mu_{jm} \, \chi_{jm}
\end{align}
where $\chi_{jm},\, j \in \N_{0},m \in \D_j$ are the characteristic functions of the $b$-adic boxes $I_{jm}$ and the sequence $\mu_{jm}$ satisfies
\[ \left( \sum_{j \in \N_0^d} b^{(j_1 + \ldots + j_d)(r - \frac{1}{p} + 1) q} \left( \sum_{m \in \D_j} | \mu_{jm}|^p \right)^{\frac{q}{p}} \right)^{\frac{1}{q}} < \infty. \]
Then analogously to \cite[Proposition 2.34]{T10a} one can prove that $f$ belongs to $S_{pq}^r B^b (\Q^d)$ and therefore to $S_{pq}^r B (\Q^d)$. To prove this let $\psi_M, \psi_F$ be real compactly supported $L_2$-normed $b$-adic Daubechies wavelets on $\R$ analogous to \cite[(1.55--1.56)]{T10a} and according to \cite[Theorem 5.1]{RW98}. We then expand the functions $\chi_{j_1 m_1}(x_1), \ldots, \chi_{j_d m_d}(x_d)$ into the wavelet representation according to \cite[(2.51--2.53)]{T10a} and insert $\chi_{jm}(x) = \chi_{j_1 m_1}(x_1) \cdot \ldots \cdot \chi_{j_d m_d}(x_d)$ into \eqref{f_given_with_chi}. We split the resulting expansions as in \cite[(2.56--2.60)]{T10a}. Then we have $2^d$ terms sorted into the cases $(j_1 \geq k_1, \ldots, j_d \geq k_d), \ldots, (j_1 < k_1, \ldots, j_d < k_d)$. The index $k = (k_1, \ldots, k_d)$ is according to \cite[(2.51)]{T10a}. We get a $b$-adic version of \cite[(2.54)]{T10a} and \cite[(2.55)]{T10a}. This guarantees counterparts of \cite[(2.62--2.66)]{T10a} and \cite[(2.73--2.74)]{T10a}. This observation leads to the norm estimate of the lemma and therefore prooves it. The next step is to estimate
\begin{align} \label{estimate_f}
\left\| f | S_{pq}^r B(\Q^d) \right\| \geq c \left( \sum_{j \in \N_{-1}^d} b^{(j_1 + \ldots + j_d)(r - \frac{1}{p} + 1) q} \left( \sum_{m \in \D_j, \, \ell \in \B_j} | \mu_{jm\ell} (f)|^p \right)^{\frac{q}{p}} \right)^{\frac{1}{q}}
\end{align}
for all $f \in S_{pq}^r B(\Q^d)$ analogously to \cite[Proposition 2.37]{T10a} ($b$-adic) where $\mu_{jm\ell}(f)$ is the sequence of the $b$-adic Haar coefficients. Finally, one gets a counterpart to \cite[Proposition 2.38]{T10a} therefore proving the theorem of this section. To do so, we respresent
\begin{align*}
   h_{jm\ell} & = \sum_{k = 0}^{b - 1} e^{\frac{2\pi i}{b} k \ell} \chi_{j+1,bm+k}, \\
h_{-1,0,1} & = \chi_{0,0}.
\end{align*}
Then every function represented as in \eqref{f_represented_h} can be represented as in \eqref{f_given_with_chi} and therefore belongs to $S_{pq}^r B(\Q^d)$. Conversely, every $f \in S_{pq}^r B(\Q^d)$ gives the estimation \eqref{estimate_f} while the representability \eqref{f_represented_h} follows from the fact that the $b$-adic Haar system is an orthonormal basis in $L_2(\Q^d)$. Therefore, one obtains the equivalence of the norms. All further technicalities can be found in the proof of \cite[Theorem 2.9]{T10a} and the references given there. The unconditionality is clear in view of \eqref{eq_quasinorm} The assertion can be obtained for $1 < p, q \leq \infty$ with $\frac{1}{p} - 1 < r < 0$ as explained in Step 2 of the proof of \cite[Proposition 2.38]{T10a}. It is also explained there how to prove the generalization of the duality. \cite[Theorem 1.20]{T10a} is here helpful as well. The remaining cases with $q < \infty$ can be obtained by real interpolation as explained in Step 3 of the proof of \cite[Proposition 2.38]{T10a} (with higher dimension not changing anything). All other cases $1 < p < \infty, \frac{1}{p} - 1 < r \leq 0, q = \infty$ can be solved by duality as well.

\section{The Haar coefficients of the generalized Hammersley type point sets}

Before we can compute the Haar coefficients we need some short calculations. We omit the proofs since they are nothing further but easy exercises.

\begin{lem} \label{lem_factor5}

For any integer $b\geq 2$ and for any $\ell \in \left\{ 1, \ldots, b - 1 \right\}$ we have
\[ \sum_{k=1}^{b-1} k e^{\frac{2\pi i}{b} \ell k} = \frac{b}{e^{\frac{2 \pi i}{b} \ell} - 1} = \sum_{k=0}^{b-2} \sum_{r = k + 1}^{b-1} e^{\frac{2\pi i}{b} r \ell}. \]

\end{lem}

\begin{lem} \label{lem_haar_coeff_besov_x}

Let $f(x) = x_1 x_2$ for $x=(x_1,x_2) \in \Q^2$. Let $j \in \N_{-1}^2,\,m\in\D_j,\ell\in\B_j$ and let $\mu_{jm\ell}$ be the $b$-Haar coefficient of $f$. Then
\begin{enumerate}[(i)]
	\item If $j = (j_1,j_2) \in \N_0^2$ then
	      \[ \mu_{jm\ell} = \frac{b^{-2j_1 - 2j_2 - 2}}{(e^{\frac{2\pi i}{b} \ell_1} - 1)(e^{\frac{2\pi i}{b} \ell_2} - 1)}. \]
	\item If $j = (j_1,-1)$ with $j_1 \in \N_0$ then
	      \[ \mu_{jm\ell} = \frac{1}{2}\frac{b^{-2j_1 - 1}}{e^{\frac{2\pi i}{b} \ell_1} - 1}. \]	
	\item If $j=(-1,j_2)$ with $j_2\in\N_0$ then
	      \[ \mu_{jm\ell} = \frac{1}{2} \frac{b^{-2j_2 - 1}}{e^{\frac{2\pi i}{b} \ell_2} - 1}. \]	
	\item If $j = (-1,-1)$ then $\mu_{jm\ell} = \frac{1}{4}$.	
\end{enumerate}

\end{lem}

\begin{lem} \label{lem_haar_coeff_besov_indicator}

Let $z = (z_1,z_2) \in \Q^2$ and $f(x) = \1_{C_z}(x)$ for $x = (x_1,x_2) \in \Q^2$. Let $j \in \N_{-1}^2,\,m \in \D_j,\ell \in \B_j$ and let $\mu_{jm\ell}$ be the Haar coefficient of $f$. Then $\mu_{jm\ell} = 0$ whenever $z$ is not contained in the interior of the $b$-adic box $I_{jm}$ supporting the functions $h_{jm\ell}$. If $z$ is contained in the interior of $I_{jm}$ then
\begin{enumerate}[(i)]
	\item If $j = (j_1,j_2) \in \N_0^2$ then there is a $k = (k_1,k_2)$ with $k_1,k_2 \in \{0,1,\ldots,b-1\}$ such that $z$ is contained in $I_{jm}^k$. Then 
	      \begin{multline*}
        \mu_{jm\ell} = b^{-j_1-j_2-2} \left[ (bm_1+k_1+1-b^{j_1+1}z_1) e^{\frac{2\pi i}{b}k_1\ell_1} + \sum_{r_1=k_1+1}^{b-1} e^{\frac{2\pi i}{b}r_1\ell_1} \right] \times\\
        \times \left[ (bm_2+k_2+1 - b^{j_2 + 1} z_2) e^{\frac{2\pi i}{b}k_2\ell_2} + \sum_{r_2 = k_2 + 1}^{b-1} e^{\frac{2\pi i}{b}r_2\ell_2} \right].
        \end{multline*}       
  \item If $j = (j_1,-1)$ with $j_1 \in \N_0$ then there is a $k_1 \in \{0,1,\ldots,b-1\}$ such that $z$ is contained in $I_{jm}^{k_1}$. Then
        \[ \mu_{jm\ell} = b^{-j_1-1} \left[ (bm_1 + k_1 + 1 - b^{j_1 + 1} z_1) e^{\frac{2\pi i}{b} k_1 \ell_1} + \sum_{r_1 = k_1 + 1}^{b-1} e^{\frac{2\pi i}{b}r_1\ell_1} \right] (1 - z_2). \]       
  \item If $j = (-1,j_2)$ with $j_2 \in \N_0$ then there is a $k_2 \in \{0,1,\ldots,b-1\}$ such that $z$ is contained in $I_{jm}^{k_2}$. Then
        \[ \mu_{jm\ell} = b^{-j_2-1} (1-z_1) \left[ (bm_2 + k_2 + 1 - b^{j_2 + 1} z_2) e^{\frac{2\pi i}{b}k_2\ell_2} + \sum_{r_2 = k_2 + 1}^{b-1} e^{\frac{2\pi i}{b}r_2 \ell_2} \right]. \]        
  \item If $j = (-1,-1)$ then $\mu_{jm\ell} = (1 - z_1)(1 - z_2)$.  
\end{enumerate}
\end{lem}
The following lemmas are the last step in the computation of the Haar coefficients.

\begin{lem} \label{lem_coeff_calc}

Let $j \in \N_0^2,\,m \in \D_j,\ell \in \B_j$ such that $j_1 + j_2 < n - 1$. Then
\begin{multline*}
\sum_{z \in \Rn \cap I_{jm}} \left[ (bm_1 + k_1 + 1 - b^{j_1 + 1} z_1) e^{\frac{2\pi i}{b}k_1\ell_1} + \sum_{r_1=k_1+1}^{b-1} e^{\frac{2\pi i}{b} r_1 \ell_1} \right] \times\\
\shoveright{\times\left[(bm_2+k_2+1-b^{j_2+1}z_2)e^{\frac{2\pi i}{b}k_2\ell_2}+\sum_{r_2=k_2+1}^{b-1}e^{\frac{2\pi i}{b}r_2\ell_2}\right]}\\
\shoveleft{= \frac{b^{n - j_1 - j_2} \pm b^{j_1 + j_2 - n + 2}}{(e^{\frac{2\pi i}{b}\ell_1} - 1)(e^{\frac{2\pi i}{b}\ell_2} - 1)}.}\\
\text{By the sign $\pm$ in the numerator we mean either $+$ or $-$ depending on $j$.}\qquad\qquad
\end{multline*}

\end{lem}

\begin{proof}

Let $z \in I_{jm}$. Then there is a $k \in \{0,1,\ldots,b-1\}^2$ such that $z \in I_{jm}^k$. We have $0 \leq m_i < b^{j_i},\,i=1,2$. Hence we can expand $m_i$ in base $b$ as
\[ m_i = b^{j_i - 1} m_1^{(i)} + b^{j_i - 2} m_2^{(i)} + \ldots + m_{j_i}^{(i)}. \]
Since $z \in \Rn \cap I_{jm}^k$ we have
\[ b^{-j_1 - 1}(bm_1 + k_1) \leq \frac{t_n}{b} + \frac{t_{n-1}}{b^2} + \ldots + \frac{t_1}{b^n} < b^{-j_1 - 1}(bm_1 + k_1 + 1). \]
Inserting the expansion of $m_1$ in the last inequality gives us
\begin{align*}
\frac{m_1^{(1)}}{b} + \frac{m_2^{(1)}}{b^2} + \ldots + \frac{m_{j_1}^{(1)}}{b^{j_1}} + \frac{k_1}{b^{j_1+1}} & \leq \frac{t_n}{b} + \frac{t_{n-1}}{b^2} + \ldots+\frac{t_1}{b^n}\\
& < \frac{m_1^{(1)}}{b} + \frac{m_2^{(1)}}{b^2} + \ldots + \frac{m_{j_1}^{(1)}}{b^{j_1}} + \frac{k_1+1}{b^{j_1+1}}.
\end{align*}
Analogously we have
\[ b^{-j_2 - 1} (bm_2+k_2) \leq \frac{s_1}{b} + \frac{s_2}{b^2} + \ldots + \frac{s_n}{b^n} < b^{-j_2 - 1}(bm_2 + k_2 + 1). \]
Hence
\begin{align*}
\frac{m_1^{(2)}}{b} + \frac{m_2^{(2)}}{b^2} + \ldots + \frac{m_{j_2}^{(2)}}{b^{j_2}} + \frac{k_2}{b^{j_2+1}} & \leq \frac{s_1}{b} + \frac{s_2}{b^2} + \ldots + \frac{s_n}{b^n}\\
& < \frac{m_1^{(2)}}{b} + \frac{m_2^{(2)}}{b^2} + \ldots + \frac{m_{j_2}^{(2)}}{b^{j_2}} + \frac{k_2+1}{b^{j_2+1}}.
\end{align*}
So one gets a characterization of the fact that $z \in \Rn\cap I_{jm}^k$ in the form
\[ t_n = m_1^{(1)},\,t_{n - 1} = m_2^{(1)},\,\ldots,\,t_{n - j_1 + 1} = m_{j_1}^{(1)},\,t_{n - j_1} = k_1 \]
and
\[ s_1 = m_1^{(2)},\,s_2 = m_2^{(2)},\,\ldots,\,s_{j_2} = m_{j_2}^{(2)},\,s_{j_2+1} = k_2. \]

Hence $t_1,t_2,\ldots,t_{j_2}$ and $t_{n-j_1+1},\ldots,t_{n-1},t_n$ are determined by the condition $z \in \Rn \cap I_{jm}$ and $t_{n-j_1}$ and $t_{j_2+1}$ are determined by $k = (k_1,k_2)$ for which $z \in I_{jm}^k$ while $t_{j_2 + 2},\ldots,t_{n - j_1 - 1}\in\{0,1,\ldots,b-1\}$ can be chosen arbitrarily. Then we calculate
\begin{align*}
& bm_1 + k_1 + 1 - b^{j_1 + 1} z_1\\
& = 1 + b^{j_1} t_n + b^{j_1 - 1} t_{n - 1} + \ldots + bt_{n - j_1 + 1} + t_{n - j_1}\\
&\qquad - b^{j_1} t_n - b^{j_1 - 1} t_{n - 1} - \ldots - b^{j_1 - n + 1} t_1\\
& = 1 - b^{-1} t_{n - j_1 - 1} - \ldots - b^{j_1 - n + 1} t_1\\
& = 1 - b^{-1} t_{n - j_1 - 1} - \ldots - b^{j_1 + j_2 - n + 2} t_{j_2 + 2} - b^{j_1+  j_2 - n + 1} t_{j_2 + 1} - \varepsilon_1
\end{align*}
where
\[ \varepsilon_1 = b^{j_1 + j_2 - n} t_{j_2} + \ldots + b^{j_1 - n + 1} t_1 \]
and
\begin{align*}
& bm_2 + k_2 + 1 - b^{j_2 + 1} z_2\\
& = 1 + b^{j_2} s_1 + b^{j_2 - 1} s_2 + \ldots + bs_{j_2} + s_{j_2 + 1}\\
&\qquad - b^{j_2} s_1 - b^{j_2-  1} s_2 - \ldots - b^{j_2 - n + 1} s_n\\
& = 1 - b^{-1} s_{j_2 + 2} - \ldots - b^{j_2 - n + 1} s_n\\
& = 1 - b^{-1} s_{j_2 - 2} - \ldots - b^{j_1 + j_2 - n + 2} s_{n - j_1 - 1} - b^{j_1 + j_2 - n + 1} s_{n - j_1} - \varepsilon_2
\end{align*}
where
\[ \varepsilon_2 = b^{j_1 + j_2 - n} s_{n - j_1 + 1} + \ldots + b^{j_1 - n + 1} s_n. \]
This means that
\[ b m_1 + k_1 + 1 - b^{j_1 + 1} z_1 = h b^{j_1 + j_2 - n + 2} - b^{j_1 + j_2 - n + 1} t_{j_2 + 1} - \varepsilon_1 \]
for $h = 1,2,\ldots,b^{n - j_1 - j_2 - 2}$. It is clear that there must be some permutation $\sigma$ of $\{ 1, 2, \ldots, b^{n - j_1 - j_2 - 2} \}$ such that
\[ b m_2 + k_2 + 1 - b^{j_2 + 1} z_2 = \sigma(h) b^{j_1 + j_2 - n + 2} - b^{j_1 + j_2 - n + 1} s_{n - j_1} - \varepsilon_2. \]
We abbreviate $X = n - j_1 - j_2 - 2$. Then
\begin{align*}
& \sum_{z \in \Rn \cap I_{jm}} \left[ (bm_1 + k_1 + 1 - b^{j_1 + 1} z_1) e^{\frac{2\pi i}{b} k_1 \ell_1} + \sum_{r_1 = k_1 + 1}^{b-1} e^{\frac{2\pi i}{b} r_1 \ell_1} \right] \times\\
&\qquad\qquad\qquad\qquad\qquad\qquad \times \left[ (bm_2 + k_2 + 1 - b^{j_2 + 1} z_2) e^{\frac{2\pi i}{b}k_2\ell_2} + \sum_{r_2 = k_2 + 1}^{b-1} e^{\frac{2\pi i}{b} r_2 \ell_2} \right]\\
& = \sum_{k_1 = 0}^{b-1} \sum_{k_2 = 0}^{b-1} \sum_{z \in \Rn \cap I_{jm}^k} \left[ \ldots \right] \times \left[ \ldots \right]\\
& = \sum_{k_1 = 0}^{b-1} \sum_{k_2 = 0}^{b-1} \sum_{h = 1}^{b^X} \left[ \left( hb^{-X} - b^{-X - 1} t_{j_2 + 1} - \varepsilon_1 \right) e^{\frac{2\pi i}{b} k_1 \ell_1} + \sum_{r_1 = k_1 + 1}^{b-1} e^{\frac{2\pi i}{b} r_1 \ell_1} \right] \times\\
&\qquad\qquad\qquad\qquad \times \left[ \left( \sigma(h) b^{-X} - b^{-X - 1} s_{n - j_1} - \varepsilon_2 \right) e^{\frac{2\pi i}{b} k_2 \ell_2} + \sum_{r_2 = k_2 + 1}^{b-1} e^{\frac{2\pi i}{b} r_2 \ell_2} \right].
\end{align*}

After having expanded the product and changed the order of summation we analyze the summands separately in a fitting order. We recall that $s_{n-j_1}$ depends on $k_1$ and $t_{j_2+1}$ depends on $k_2$. Except the last two, all summands are equal to zero because each has the sum of unity roots as a factor. The nonzero summands are
\[ \sum_{h=1}^{b^X} \sum_{k_1 = 0}^{b-1} \sum_{r_1 = k_1 + 1}^{b-1} e^{\frac{2\pi i}{b} r_1 \ell_1} \sum_{k_2 = 0}^{b-1} \sum_{r_2 =k _2 + 1}^{b-1} e^{\frac{2\pi i}{b} r_2 \ell_2} = \frac{b^{n - j_1 - j_2}}{(e^{\frac{2\pi i}{b} \ell_1} - 1)(e^{\frac{2\pi i}{b} \ell_2} - 1)} \]
(by Lemma \ref{lem_factor5}) and
\begin{multline*}
\sum_{h=1}^{b^X} \sum_{k_1 = 0}^{b-1} \sum_{k_2 = 0}^{b-1} b^{-X - 1} t_{j_2 + 1} b^{-X - 1} s_{n - j_1} e^{\frac{2\pi i}{b} k_1 \ell_1} e^{\frac{2\pi i}{b} k_2 \ell_2}\\
= b^{j_1 + j_2 - n} \sum_{k_1 = 0}^{b-1} s_{n - j_1} e^{\frac{2\pi i}{b} k_1 \ell_1} \sum_{k_2 = 0}^{b-1} t_{j_2 + 1} e^{\frac{2\pi i}{b} k_2 \ell_2}.
\end{multline*}
We know that $t_{n - j_1} = k_1$ and that either $s_i = t_i$ or $s_i = b - 1 - t_i$ for all $i = 1,\ldots,n$. Hence $s_{n - j_1}$ is either $k_1$ or $b - 1 - k_1$. Since
\[ \sum_{k_1 = 0}^{b-1} (b-1) e^{\frac{2\pi i}{b} k_1 \ell_1} = 0 \]
we have
\begin{align} \label{sign_term}
\sum_{k_1 = 0}^{b-1} s_{n - j_1} e^{\frac{2\pi i}{b} k_1 \ell_1} = \pm \frac{b}{e^{\frac{2\pi i}{b} \ell_1} - 1}
\end{align}
using Lemma \ref{lem_factor5} and the sign depends on $j_1$. Also we know that $s_{j_2 + 1} = k_2$ and that either $s_{j_2 + 1} = t_{j_2 + 1}$ or $s_{j_2} = b - 1 - t_{j_2 + 1}$. Hence
\[ \sum_{k_2 = 0}^{b-1} t_{j_2 + 1} e^{\frac{2\pi i}{b} k_2 \ell_2} = \pm \frac{b}{e^{\frac{2\pi i}{b} \ell_2} - 1} \]
and the sign depends on $j_2$. So alltogether our last summand is
\[ b^{j_1 + j_2 - n} \frac{\pm b^2}{(e^{\frac{2\pi i}{b} \ell_1} - 1) (e^{\frac{2\pi i}{b} \ell_2} - 1)} = \frac{\pm b^{j_1 + j_2 - n + 2}}{(e^{\frac{2\pi i}{b} \ell_1} - 1)(e^{\frac{2\pi i}{b} \ell_2} - 1)} \]
and the sign depends on $j$. Adding both summands which are nonzero gives us the stated result. One can find a longer though straightforward version of the calculation in \cite{M12}

\end{proof}

\begin{lem} \label{lem_xn}

Let
\[ x_n := \sum_{t_1,\ldots,t_n = 0}^{b-1}\sum_{j=1}^n b^{-j} t_j \]
and
\[ y_n := \sum_{t_1,\ldots,t_n = 0}^{b-1} \sum_{i=1}^n b^i t_i \]
for any positive integer $n$. Then
\[ x_n = \frac{1}{2}(b^n - 1) \]
and
\[ y_n = b^{n + 1} x_n = \frac{1}{2} b^{n+1} (b^n - 1). \]

\end{lem}

\begin{proof}

Clearly, $x_1 = \frac{1}{2} (b-1)$ and inductively
\begin{align*}
x_n & = \sum_{t_n} \sum_{t_1,\ldots,t_{n-1}} \sum_{j=1}^{n-1} b^{-j} t_j + b^{-n} \sum_{t_1,\ldots,t_{n-1}} \sum_{t_n} t_n\\
    & = b \, x_{n-1} + b^{-n} \, b^{n-1} \, \frac{b \, (b-1)}{2}\\
    & = b \, \frac{1}{2} \, (b^{n-1} - 1) + \frac{1}{2} \, (b-1)\\
    & = \frac{1}{2} \, (b^n-1).
\end{align*}
One sees that $y_n = b^{n+1} x_n$ simply by checking that
\[ \sum_{i=1}^n b^i t_i = b^{n+1} \sum_{i=1}^n b^{i - n - 1} t_i = b^{n+1} \sum_{i=1}^n b^{-i} t_{n + 1 - i}. \]
Summing over $t_1,\ldots,t_n$ will give us $y_n$ on the left side. On the right side it will give us $b^{n+1} x_n$ although the order of the $t_i$ is reversed with respect to the definition of the numbers $x_n$.

\end{proof}

We will use this fact that the order of the $t_i$ is irrelevant in further proofs. But not only the order is irrelevant but even the concrete index of the $t_j$. For example the value of
\[ \sum_{t_{n+1},\ldots,t_{2n} = 0}^{b-1} \sum_{j=1}^n b^{-j} t_{j + n} \]
is the same as the value of $x_n$.

\begin{lem} \label{lem_zn}

Let
\[ z_n := \sum_{t_1,\ldots,t_n = 0}^{b-1} \sum_{i,j=1}^n b^{i-j} t_i t_j \]
for any positive integer $n$. Then
\[ z_n = \frac{1}{4} b^{2n+1} + \frac{n}{12} b^{n+2} - \frac{1}{2} b^{n+1} - \frac{n}{12} b^n + \frac{1}{4} b. \]

\end{lem}

The proof is analogous to above.

\begin{lem} \label{lem_sum_z}

Let $z=(z_1,z_2)$. Then
\[ \sum_{z \in \Rn} (1 - z_1) (1 - z_2) = 1 + b^{-n - 1} \sum_{t_1,\ldots,t_n}^{b-1} \sum_{i,j=1}^n b^{i-j} t_i s_j. \]

\end{lem}

\begin{proof}

We first calculate for some $z \in \Rn$
\begin{align*}
(1 - z_1) (1 - z_2) & = (1 - b^{-1} t_n - \ldots -b ^{-n} t_1) (1 - b^{-1} s_1 - \ldots - b^{-n} s_n)\\
                    & = 1 - b^{-1} t_n - \ldots - b^{-n} t_1 - b^{-1} s_1 - \ldots - b^{-n} s_n +\\
                    &\qquad + \sum_{i,j=1}^n b^{-n + i - j - 1} t_i s_j.
\end{align*}

Now we sum over all $z \in \Rn$ which corresponds to summing over all $t_1,\ldots,t_n \in\{0,1,\ldots,b-1\}$ and get
\begin{align*}
& \sum_{z \in \Rn \cap I_{(-1,-1),(0,0)}} (1 - z_1) (1 - z_2)\\
& = \sum_{t_1,\ldots,t_n} \left(1 - b^{-1} t_n - \ldots - b^{-n} t_1 - b^{-1} s_1 - \ldots - b^{-n} s_n + b^{-n - 1}\sum_{i,j=1}^n b^{i-j} t_i s_j \right)\\
& = b^n - b^{-1} \, b^{n-1} \sum_{t_n = 0}^{b-1} t_n - b^{-1} \, b^{n-1} \sum_{t_1 = 0}^{b-1} s_1 - \ldots - b^{-n} \, b^{n-1} \sum_{t_1 = 0}^{b-1} t_1 -\\
&\qquad\qquad -b^{-n} \, b^{n-1} \sum_{t_n = 0}^{b-1} s_n + b^{-n-1} \sum_{t_1,\ldots,t_n} \sum_{i,j=1}^n b^{i-j} t_i s_j\\
& = b^n - 2 \left( b^{n-2} \, \frac{1}{2} \, (b-1) \, b + \ldots + b^{-1} \, \frac{1}{2} \, (b-1) \, b \right) + b^{-n - 1} \sum_{t_1,\ldots,t_n} \sum_{i,j=1}^n b^{i-j} t_i s_j\\
& = b^n - (b-1) (b^{n-1} + \ldots + 1) + b^{-n - 1} \sum_{t_1,\ldots,t_n} \sum_{i,j=1}^n b^{i-j} t_i s_j\\
& = 1 + b^{-n - 1} \sum_{t_1,\ldots,t_n} \sum_{i,j=1}^n b^{i-j} t_i s_j\\
\end{align*}

\end{proof}

\begin{lem} \label{lem_a_ident}

We consider a generalized Hammersley type point set $\Rn$. Then
\[ \sum_{t_1,\ldots,t_n = 0}^{b-1} \sum_{i,j=1}^n b^{i-j} t_i s_j = \frac{1}{4} b^{2n+1} - \frac{1}{2} b^{n+1} + \frac{1}{4} b + (2a_n - n) \frac{b^2 - 1}{12} b^n. \]

\end{lem}

\begin{proof}

For better readability we write $a$ instead of $a_n$. We can assume that $s_1 = t_1,\ldots,s_a = t_a, s_{a+1} = b - 1 - t_{a+1},\ldots,s_n = b - 1 - t_n$. Otherwise we would have to rename the $t_j$. This assumption allows us to split the sum in a compact way. So,
\begin{align*}
& \sum_{i,j=1}^n b^{i-j}t_i s_j = \sum_{i,j=1}^a b^{i-j} t_i t_j + \sum_{i=1}^a \sum_{j=a+1}^n b^{i-j} t_i (b - 1 - t_j) +\\
                              &\qquad\qquad + \sum_{i=a+1}^n \sum_{j=1}^a b^{i-j} t_i t_j + \sum_{i,j = a+1}^n b^{i-j} t_i (b - 1 - t_j)\\
                              & = \sum_{i,j=1}^a b^{i-j} t_i t_j + (b-1) \sum_{i=1}^a \sum_{j=a+1}^n b^{i-j} t_i - \sum_{i=1}^a \sum_{j=a+1}^n b^{i-j} t_i t_j +\\
                              &\qquad + \sum_{i=a+1}^n \sum_{j=1}^a b^{i-j} t_i t_j + (b-1)\sum_{i=a+1}^n \sum_{j=a+1}^n b^{i-j} t_i - \sum_{i=a+1}^n \sum_{j=a+1}^n b^{i-j} t_i t_j.
\end{align*}
Summing over $t_1,\ldots,t_n$ and analyzing every term separately will give us
\begin{align*}
& \sum_{t_1,\ldots,t_n} \sum_{i,j=1}^a b^{i-j} t_i t_j = b^{n-a} z_a,\\
\intertext{as well as using $y_n = b^{n+1} x_n$}
& \sum_{t_1,\ldots,t_n}(b-1) \sum_{i=1}^a \sum_{j=a+1}^n b^{i-j} t_i = (b-1) b^{n-a} y_a \sum_{j=a+1}^n b^{-j}\\
& \qquad = b^{n+1} x_a (b^{-a} - b^{-n}),\\
\intertext{and}
\sum_{t_1,\ldots,t_n}\sum_{i=1}^a \sum_{j=a+1}^n b^{i-j} t_i t_j & = \sum_{t_1,\ldots,t_a} \sum_{i=1}^a b^i t_i \sum_{t_{a+1},\ldots,t_n} \sum_{j=a+1}^n b^{-j} t_j\\
                                                                 & = y_a \sum_{t_{a+1},\ldots,t_n} b^{-a} \sum_{j=a+1}^n b^{a-j} t_j = x_a x_{n-a} b,
\end{align*}
since we have already seen that the indexes of $t_j$ are irrelevant. We also get with a similar argumentation
\begin{align*}
& \sum_{t_1,\ldots,t_n} \sum_{i=a+1}^n \sum_{j=1}^a b^{i-j} t_i t_j = \sum_{t_1,\ldots,t_a} \sum_{j=1}^a b^{-j} t_j \sum_{t_{a+1},\ldots,t_n} \sum_{i=a+1}^n b^i t_i\\
& = \sum_{t_1,\ldots,t_a} \sum_{j=1}^a b^{-j} t_j \sum_{t_{a+1},\ldots,t_n} b^a \sum_{i=a+1}^n b^{i-a} t_i = x_a b^a y_{n-a} = x_a x_{n-a} b^{n+1},\\
& \sum_{t_1,\ldots,t_n} (b-1) \sum_{i=a+1}^n \sum_{j=a+1}^n b^{i-j} t_i=(b-1) b^a \sum_{t_{a+1},\ldots,t_n} \sum_{i=a+1}^n b^i t_i \sum_{j=a+1}^n b^{-j}\\
& = b^a y_{n-a} b^a (b^{-a} - b^{-n}) = x_{n-a} (b^{n+1} - b^{a+1})\\
\intertext{and}
& \sum_{t_1,\ldots,t_n} \sum_{i=a+1}^n \sum_{j=a+1}^n b^{i-j} t_i t_j = b^a\sum_{t_{a+1},\ldots,t_n} \sum_{i=a+1}^n \sum_{j=a+1}^n b^{(i-a)+(a-j)} t_i t_j = b^a z_{n-a}.
\end{align*}
So what we have is
\begin{multline*}
\sum_{t_1,\ldots,t_n}^{b-1} \sum_{i,j=1}^n b^{i-j} t_i s_j\\
= b^{n-a} z_a - b^a z_{n-a} + x_a b(b^{n-a} - 1) + x_a x_{n-a} b (b^n-1) + x_{n-a} b^{a+1}(b^{n-a} - 1).
\end{multline*}

Inserting the values of $z_a,\,z_{n-a},\,x_a,$ and $x_{n-a}$ and simplifying will give us the stated assertion.

\end{proof}

\begin{prp} \label{prp_minus1}

Let $\mu_{jm\ell}$ be the $b$-adic Haar coefficients of the discrepancy function of $\Rn$. Then
\[ \mu_{(-1,-1),(0,0),(1,1)} = \frac{1}{4} b^{-2n} + \frac{1}{2} b^{-n} + (2a_n - n) \frac{b^2 - 1}{12} b^{-n - 1}. \]

\end{prp}

\begin{proof}

Using the last lemma we have
\[ \sum_{t_1,\ldots,t_n = 0}^{b-1} \sum_{i,j=1}^n b^{i-j} t_i s_j = \frac{1}{4} b^{2n+1} - \frac{1}{2} b^{n+1} + \frac{1}{4} b + (2a_n - n) \frac{b^2 - 1}{12} b^n. \]
Hence using Lemmas \ref{lem_haar_coeff_besov_x}, \ref{lem_haar_coeff_besov_indicator} and \ref{lem_sum_z}
\begin{align*}
& \mu_{(-1,-1),(0,0),(1,1)} = b^{-n} \sum_{z \in \Rn} (1 - z_1) (1 - z_2) - \frac{1}{4}\\
                          & = b^{-n} \left( 1 + b^{-n - 1} \left( \frac{1}{4} b^{2n+1} - \frac{1}{2} b^{n+1} + \frac{1}{4} b + (2a_n - n) \, \frac{b^2 - 1}{12} \, b^n \right) \right) - \frac{1}{4}\\
                          & = \frac{1}{4} b^{-2n} + \frac{1}{2} b^{-n} + (2a_n - n) \, \frac{b^2 - 1}{12} \, b^{-n - 1}.
\end{align*}

\end{proof}

\begin{lem} \label{lem_middle}

Let $j = (j_1,-1)$ for $j_1 \in \N_0$ with $j_1 \leq n-1$, $m = (m_1,0)$ with $0 \leq m_1 < b^{j_1}$ and $\ell = (\ell_1,1)$ with $1 \leq \ell_1 < b$. Then
\begin{multline*}
\sum_{z \in \Rn \cap I_{jm}} \left[ (bm_1 + k_1 + 1 - b^{j_1 + 1} z_1) e^{\frac{2\pi i}{b} k_1 \ell_1} + \sum_{r_1 = k_1 + 1}^{b-1} e^{\frac{2\pi i}{b} r_1 \ell_1} \right] (1 - z_2)\\
= \frac{b^{n-j_1}(1 - 2\varepsilon) \mp b^{j_1-n+1}}{2(e^{\frac{2\pi i}{b} \ell_1} - 1)} + \frac{w_{j_1}}{(e^{\frac{2\pi i}{b} \ell_1} - 1)^2},
\end{multline*}
where $w_{j_1}$ is either $e^{\frac{2\pi i}{b} \ell_1}$ or $-1$, the sign of $\mp$ depends on $j_1$ and we have $\varepsilon b^{n - j_1} \leq b$.

An analogous result holds for $j = (-1,j_2)$ where $j_2 \in \N_0$ with $j_2 \leq n - 1$, $m = (0,m_2)$ with $0 \leq m_2 < b^{j_2}$ and $\ell = (1,\ell_2)$ with $1 \leq \ell_2 < b$.
\end{lem}

\begin{proof}

Let $z \in \Rn\cap I_{jm}$. Then there is a $k = (k_1,-1)$, $k_1 \in \{ 0, 1, \ldots, b - 1 \}$ such that, $z \in \Rn \cap I_{jm}^k$. We use the methods from from Lemma \ref{lem_coeff_calc} for the proof. We have
\[ bm_1 + k_1 + 1 - b^{j_1 + 1} z_1 = 1 - b^{-1} t_{n - j_1 - 1} - \ldots - b^{j_1 - n + 1} t_1 \]
which means that
\[ bm_1 + k_1 + 1 - b^{j_1 + 1} z_1 = h b^{j_1 - n + 1} \]
for $h = 1, 2, \ldots, b^{n - j_1 - 1}$. The numbers $t_{n - j_1 + 1},\ldots,t_n$ are determined by the condition $z \in \Rn \cap I_{jm}$ and $t_{n - j_1} = k_1$. All other $t_j$ can be chosen arbitrarily. We also have
\[ 1 - z_2 = 1 - b^{-1} s_1 - \ldots - b^{j_1 - n + 1} s_{n - j_1 - 1} - b^{j_1 - n} s_{n - j_1} - \varepsilon \]
where $\varepsilon = b^{j_1 - n - 1} s_{n - j_1 + 1} + \ldots + b^{-n} s_n$. Clearly, $\varepsilon b^{n - j_1} \leq b$.

So there must be a permutation $\sigma$ such that
\[ 1 - z_2 = \sigma(h) b^{j_1 - n + 1} - b^{j_1 - n} s_{n - j_1} - \varepsilon. \]
Hence
\begin{align*}
& \sum_{z \in \Rn \cap I_{jm}} \left[ (bm_1 + k_1 + 1 - b^{j_1 + 1} z_1) e^{\frac{2\pi i}{b} k_1 \ell_1} + \sum_{r_1 = k_1 + 1}^{b-1} e^{\frac{2\pi i}{b} r_1 \ell_1} \right] (1 - z_2)\\
& = \sum_{k_1 = 0}^{b-1} \sum_{h=1}^{b^{n - j_1 - 1}} \left[hb^{j_1 - n + 1} e^{\frac{2\pi i}{b} k_1 \ell_1} + \sum_{r_1 = k_1 + 1}^{b-1} e^{\frac{2\pi i}{b} r_1 \ell_1} \right] (\sigma(h) b^{j_1 - n + 1} - b^{j_1 - n} s_{n - j_1} - \varepsilon)\\
\end{align*}

We analyze the summands separately after having expanded the product and changed the order of summation. A longer though straightforward calculation can be found in \cite{M12}. We have
\begin{align*}
 & \sum_{h=1}^{b^{n - j_1 - 1}} h \sigma(h) b^{j_1 - n + 1} b^{j_1 - n + 1} \sum_{k_1 = 0}^{b-1} e^{\frac{2\pi i}{b} k_1 \ell_1} = 0,\\
-& \sum_{h=1}^{b^{n - j_1 - 1}} h b^{j_1 - n + 1} b^{j_1 - n} \sum_{k_1 = 0}^{b-1} s_{n - j_1} e^{\frac{2\pi i}{b} k_1 \ell_1} = \mp \frac{b^{j_1 - n + 1} + 1}{2(e^{\frac{2\pi i}{b} \ell_1} - 1)},\\
\intertext{using the equation \eqref{sign_term} from the proof for Lemma \ref{lem_coeff_calc},}\\
-& \varepsilon \sum_{h=1}^{b^{n - j_1 - 1}} h b^{j_1 - n + 1} \sum_{k_1 = 0}^{b-1} e^{\frac{2\pi i}{b} k_1 \ell_1} = 0,\\
 & \sum_{h=1}^{b^{n - j_1 - 1}} \sigma(h) b^{j_1 - n + 1} \sum_{k_1 = 0}^{b-1} \sum_{r_1 = k_1 + 1}^{b-1} e^{\frac{2\pi i}{b} r_1 \ell_1} = \frac{b^{n - j_1} + b}{2(e^{\frac{2\pi i}{b} \ell_1} - 1)},\\
-& \varepsilon \sum_{h=1}^{b^{n - j_1 - 1}} \sum_{k_1 = 0}^{b-1} \sum_{r_1 = k_1 + 1}^{b-1} e^{\frac{2\pi i}{b} r_1 \ell_1} = \frac{-\varepsilon b^{n - j_1}}{e^{\frac{2\pi i}{b} \ell_1} - 1}.\\
\intertext{and}
-& \sum_{h=1}^{b^{n - j_1 - 1}} b^{j_1 - n} \sum_{k_1 = 0}^{b-1} s_{n - j_1} \sum_{r_1 = k_1 + 1}^{b-1} e^{\frac{2\pi i}{b} r_1 \ell_1}\\
\end{align*}

For the last term we use the fact that $s_{n - j_1}$ is either $k_1$ or $b - 1 - k_1$. In the first case we have
\begin{align*}
& \sum_{k_1 = 0}^{b-1} k_1 \sum_{r_1 = k_1 + 1}^{b-1} e^{\frac{2\pi i}{b} r_1 \ell_1}\\
& = \sum_{k_1 = 1}^{b-2} k_1 \frac{1 - e^{\frac{2\pi i}{b} (k_1 + 1) \ell_1}}{e^{\frac{2\pi i}{b} \ell_1} - 1}\\
& = \frac{1}{e^{\frac{2\pi i}{b} \ell_1} - 1} \left( \frac{1}{2} (b-2) (b-1) - \sum_{k_1 = 2}^{b-1} (k_1 - 1)e^{\frac{2\pi i}{b} k_1 \ell_1} \right)\\
& = \frac{1}{e^{\frac{2\pi i}{b} \ell_1} - 1} \left( \frac{1}{2} (b-2) (b-1) - \left( \frac{b}{e^{\frac{2\pi i}{b} \ell_1} - 1} - e^{\frac{2\pi i}{b} \ell_1} \right) + \left( 0 - 1 - e^{\frac{2\pi i}{b} \ell_1} \right) \right)\\
& = \frac{1}{e^{\frac{2\pi i}{b} \ell_1} - 1} \left( \frac{b^2 - 3b}{2} - \frac{b}{e^{\frac{2\pi i}{b} \ell_1} - 1} \right)\\
& = \frac{(b-3) b}{2(e^{\frac{2\pi i}{b} \ell_1} - 1)} - \frac{b}{(e^{\frac{2\pi i}{b} \ell_1} - 1)^2}.
\end{align*}
In the other case we have
\begin{align*}
& \sum_{k_1 = 0}^{b-1} (b - 1 - k_1) \sum_{r_1 = k_1 + 1}^{b-1} e^{\frac{2\pi i}{b} r_1 \ell_1}\\
& = (b-1) \sum_{k_1 = 0}^{b-1} \sum_{r_1 = k_1 + 1}^{b-1}e^{\frac{2\pi i}{b} r_1 \ell_1} - \sum_{k_1 = 0}^{b-1} k_1 \sum_{r_1 = k_1 + 1}^{b-1} e^{\frac{2\pi i}{b} r_1 \ell_1}\\
& = \frac{(b-1)b}{(e^{\frac{2\pi i}{b} \ell_1} - 1)} - \frac{(b-3)b}{2(e^{\frac{2\pi i}{b} \ell_1} - 1)} + \frac{b}{(e^{\frac{2\pi i}{b} \ell_1} - 1)^2}\\
& = \frac{b(b+1)}{2(e^{\frac{2\pi i}{b} \ell_1} - 1)} + \frac{b}{(e^{\frac{2\pi i}{b} \ell_1} - 1)^2}.
\end{align*}
So the last term is either
\[ \frac{1}{(e^{\frac{2\pi i}{b} \ell_1} - 1)^2} - \frac{b-3}{2(e^{\frac{2\pi i}{b} \ell_1} - 1)} \]
or
\[ -\frac{b+1}{2(e^{\frac{2\pi i}{b} \ell_1} - 1)} - \frac{1}{(e^{\frac{2\pi i}{b} \ell_1} - 1)^2}. \]
Now combining the results we get in the case $s_{n - j_1} = k_1$
\begin{align*}
& \sum_{z \in \Rn \cap I_{jm}} \left[ (bm_1 + k_1 + 1 - b^{j_1 + 1} z_1) e^{\frac{2\pi i}{b} k_1 \ell_1} + \sum_{r_1 = k_1 + 1}^{b-1} e^{\frac{2\pi i}{b} r_1 \ell_1} \right] (1 - z_2)\\
& = \frac{b^{n-j_1}(1 - 2\varepsilon) - b^{j_1-n+1}}{2(e^{\frac{2\pi i}{b} \ell_1} - 1)} + \frac{e^{\frac{2\pi i}{b} \ell_1}}{(e^{\frac{2\pi i}{b} \ell_1} - 1)^2}
\end{align*}
while in the case $s_{n - j_1} = b - 1 - k_1$
\begin{align*}
& \sum_{z \in \Rn \cap I_{jm}} \left[ (bm_1 + k_1 + 1 - b^{j_1 + 1} z_1) e^{\frac{2\pi i}{b} k_1 \ell_1} + \sum_{r_1 = k_1 + 1}^{b-1} e^{\frac{2\pi i}{b} r_1 \ell_1} \right] (1 - z_2)\\
& = \frac{b^{n-j_1}(1 - 2\varepsilon) + b^{j_1-n+1}}{2(e^{\frac{2\pi i}{b} \ell_1} - 1)} - \frac{1}{(e^{\frac{2\pi i}{b} \ell_1} - 1)^2}
\end{align*}
as stated by the lemma.

\end{proof}

\section{Proof of the main result}

\begin{prp} \label{prp_haar_coeff}

Let $\Rn$ be a generalized Hammersley type point set and let $\mu_{jm\ell}$ be the $b$-adic Haar coefficient of the discrepancy function of $\Rn$ for $j \in \N_{-1}^2, \, m \in \D_j$ and $\ell \in \B_j$. Then
\begin{enumerate}[(i)]
	\item if $j \in \N_0^2$ and $j_1 + j_2 < n-1$ then
\[ \left| \mu_{jm\ell} \right| = \frac{b^{-2n}}{\left|e^{\frac{2\pi i}{b} \ell_1} - 1\right|\left|e^{\frac{2\pi i}{b} \ell_2} - 1\right|}, \] \label{prp_1}
	\item if $j \in \N_0^2$, $j_1 + j_2 \geq n-1$ and $j_1,j_2 \leq n$ then $\left| \mu_{jm\ell} \right| \leq c b^{-n - j_1 - j_2}$ for some constant $c > 0$ and
\[ \left| \mu_{jm\ell} \right| = \frac{b^{-2j_1 - 2j_2 - 2}}{\left|e^{\frac{2\pi i}{b} \ell_1} - 1\right|\left|e^{\frac{2\pi i}{b} \ell_2} - 1\right|} \]
for all but $b^n$ coefficients $\mu_{jm\ell}$, \label{prp_2}
	\item if $j \in \N_0^2$ and $j_1 \geq n$ or $j_2 \geq n$ then
\[ \left| \mu_{jm\ell} \right| = \frac{b^{-2j_1 - 2j_2 - 2}}{\left|e^{\frac{2\pi i}{b} \ell_1} - 1\right|\left|e^{\frac{2\pi i}{b} \ell_2} - 1\right|}, \] \label{prp_3}
	\item if $j = (j_1,-1)$ with $j_1 \in \N_0$ and $j_1 < n$ then $\left| \mu_{jm\ell} \right| \leq c b^{-n-j_1}$ for some constant $c > 0$, \label{prp_4}
	\item if $j = (-1,j_2)$ with $j_2 \in \N_0$ and $j_2 < n$ then $\left| \mu_{jm\ell} \right| \leq c b^{-n-j_2}$ for some constant $c > 0$, \label{prp_4a}
	\item if $j = (j_1,-1)$ with $j_1 \in \N_0$ and $j_1 \geq n$ then
\[ \left| \mu_{jm\ell} \right| = \frac{1}{2}\frac{b^{-2j_1 - 1}}{\left|e^{\frac{2\pi i}{b} \ell_1} - 1 \right|}, \] \label{prp_5}
  \item if $j = (-1,j_2)$ with $j_2 \in \N_0$ and $j_2 \geq n$ then
\[ \left| \mu_{jm\ell} \right| = \frac{1}{2}\frac{b^{-2j_2 - 1}}{\left|e^{\frac{2\pi i}{b} \ell_2} - 1 \right|}, \] \label{prp_5a}
	\item $\left| \mu_{(-1,-1),(0,0),(1,1)} \right| = | \frac{1}{4} b^{-2n} + \left( \frac{1}{2} + (2a_n - n) \frac{b - b^{-1}}{12} \right) b^{-n} |.$
\end{enumerate}

\end{prp}

\begin{proof}

Let $j \in \N_{-1}^2$ such that $j_1 \geq n$ or $j_2 \geq n$. Then there is no point of $\Rn$ which is contained in the interior of the $b$-adic box $I_{jm}$. Thereby \eqref{prp_3}, \eqref{prp_5} and \eqref{prp_5a} follow from Lemma \ref{lem_haar_coeff_besov_x} and Lemma \ref{lem_haar_coeff_besov_indicator}.

The set $\Rn$ contains $N = b^n$ points and, for fixed $j \in \N_{-1}^2$, the interiors of the $b$-adic boxes $I_{jm}$ are mutually disjoint. Therefore there are no more than $b^n$ $b$-adic boxes which contain a point of $\Rn$. This gives us the second part of \eqref{prp_2}. The first part of \eqref{prp_2} follows from Lemma \ref{lem_haar_coeff_besov_x} and Lemma \ref{lem_haar_coeff_besov_indicator} because the remaining boxes contain exactly one point of $\Rn$.

The part \eqref{prp_1} follows from Lemmas \ref{lem_haar_coeff_besov_x}, \ref{lem_haar_coeff_besov_indicator} and \ref{lem_coeff_calc}.

The last part is actually Proposition \ref{prp_minus1}.

Finally \eqref{prp_4} (and analogously \eqref{prp_4a}) follows from Lemma \ref{lem_middle} combined with Lemma \ref{lem_haar_coeff_besov_x} and Lemma \ref{lem_haar_coeff_besov_indicator}. We get
\[ \left| \mu_{jm\ell} \right| = \left| \frac{b^{-n-j_1-1}(w_{j_1} - \varepsilon \, b^{n-j_1} \, (e^{\frac{2\pi i}{b} \ell_1} - 1))}{(e^{\frac{2\pi i}{b} \ell_1} - 1)^2} \pm \frac{b^{-2n}}{2(e^{\frac{2\pi i}{b} \ell_1} - 1)} \right| \]
where $w_{j_1}$ is either $e^{\frac{2\pi i}{b} \ell_1}$ or $-1$. Clearly,
\[ \left| w_{j_1} - \varepsilon \, b^{n-j_1} \, (e^{\frac{2\pi i}{b} \ell_1} - 1) \right| \leq c. \]
for some constant $c > 0$ since $\varepsilon b^{n-j_1} \leq b$. Hence
\[ \left| \mu_{jm\ell} \right| \leq \bar{c} \, b^{-n-j_1}. \]
\end{proof}
Now we are ready to prove the main result.

\begin{proof}[Proof of Theorem \ref{thm_hammersley_disc}]

Let $\Rn$ be a generalized Hammersley type point set with $a_n$ satisfying $|2a_n - n| \leq c_0$ for some constant $c_0 \geq 0$. Let $\mu_{jm\ell}$ be the $b$-adic Haar coefficients of the discrepancy function of $\Rn$. Theorem \ref{thm_besov_char} gave us an equivalent quasi-norm on $S_{pq}^r B(\Q^2)$ so that the proof of the inequality
\[ \left( \sum_{j\in\N_{-1}^2} b^{(j_1 + j_2)(r - \frac{1}{p} + 1) q} \left( \sum_{m\in\D_j,\,\ell\in\B_j} |\mu_{jm\ell}|^p \right)^{\frac{q}{p}} \right)^{\frac{1}{q}} \leq C \, b^{n(r-1)} n^{\frac{1}{q}} \]
for some constant $C > 0$ establishes the proof of the theorem.

We use different parts of Proposition \ref{prp_haar_coeff} after having split the sum by Minkowski's inequality. We have
\begin{align*}
& \left( \sum_{j\in\N_0^2; \, j_1 + j_2 < n-1} b^{(j_1 + j_2)(r - \frac{1}{p} + 1) q} \left( \sum_{m \in \D_j, \, \ell \in \B_j} | \mu_{jm\ell}|^p \right)^{\frac{q}{p}} \right)^{\frac{1}{q}}\\
& \leq c_1 \left( \sum_{j\in\N_0^2; \, j_1 + j_2 < n-1} b^{(j_1 + j_2)(r - \frac{1}{p} + 1) q} \left( \sum_{m \in \D_j} b^{-2np} \right)^{\frac{q}{p}} \right)^{\frac{1}{q}}\\
& = c_1 \left( \sum_{j\in\N_0^2; \, j_1 + j_2 < n-1} b^{\left[ (j_1 + j_2)(r + 1) - 2n \right] q} \right)^{\frac{1}{q}}\\
& = c_1 \left( \sum_{\lambda = 0}^{n-2} b^{\left[ \lambda (r+1) - 2n \right] q} (\lambda + 1) \right)^{\frac{1}{q}}\\
& \leq c_1 n^{\frac{1}{q}} \left( \sum_{\lambda = 0}^{n-2} b^{\left[ \lambda (r+1) - 2n \right] q} \right)^{\frac{1}{q}}\\
& \leq c_2 n^{\frac{1}{q}} b^{n (r-1)}
\end{align*}
from \eqref{prp_1}. From \eqref{prp_2} we have (using the fact that $\frac{1}{p} - r > 0$)
\begin{align*}
& \left( \sum_{0 \leq j_1, j_2 \leq n; \, j_1 + j_2 \geq n-1} b^{(j_1 + j_2)(r - \frac{1}{p} + 1)q} \left( \sum_{m \in \D_j, \, \ell \in \B_j} | \mu_{jm\ell}|^p \right)^{\frac{q}{p}} \right)^{\frac{1}{q}}\\
& \leq c_3 \left( \sum_{0 \leq j_1, j_2 \leq n; \, j_1 + j_2 \geq n-1} b^{(j_1 + j_2)(r - \frac{1}{p} + 1) q} \, b^{n \frac{q}{p}} \, b^{(-n - j_1 - j_2) q} \right)^{\frac{1}{q}}\\
& \quad + c_4 \left( \sum_{0 \leq j_1, j_2 \leq n; \, j_1 + j_2 \geq n-1} b^{(j_1 + j_2)(r - \frac{1}{p} + 1)q} \, b^{(j_1 + j_2)\frac{q}{p}} \, b^{(-2j_1 - 2j_2) q} \right)^{\frac{1}{q}}\\
& = c_3 \left( \sum_{0 \leq j_1, j_2 \leq n; \, j_1 + j_2 \geq n-1} b^{\left[ (j_1 + j_2) (r - \frac{1}{p}) + \frac{n}{p} - n \right] q} \right)^{\frac{1}{q}}\\
& \quad + c_4 \left( \sum_{0 \leq j_1, j_2 \leq n; \, j_1 + j_2 \geq n-1} b^{(j_1 + j_2)(r - 1)q} \right)^{\frac{1}{q}}\\
& = c_3 \left( \sum_{\lambda = n-1}^{2n} (2n - \lambda + 1) b^{\left[ \lambda(r - \frac{1}{p}) + \frac{n}{p} - n \right] q} \right)^{\frac{1}{q}}\\
& \quad + c_4 \left( \sum_{\lambda = n-1}^{2n} (2n - \lambda + 1) b^{\lambda (r-1) q} \right)^{\frac{1}{q}}\\
& = c_3 b^{\frac{n}{p} - n} \left( \sum_{\lambda = 1}^{n + 2} \lambda b^{\left[ (2n + 1 - \lambda)(r-\frac{1}{p}) \right] q} \right)^{\frac{1}{q}} + c_4 \left( \sum_{\lambda = 1}^{n+2} \lambda b^{(2n + 1 - \lambda) (r-1) q} \right)^{\frac{1}{q}}\\
& \leq c_5 b^{n(r-1) + n(r-\frac{1}{p})} \left( \sum_{\lambda = 1}^{n + 2} \lambda b^{\lambda (\frac{1}{p} - r) q} \right)^{\frac{1}{q}} + c_6 b^{2n (r-1)} \left( \sum_{\lambda = 1}^{n+2} \lambda b^{\lambda (1-r) q} \right)^{\frac{1}{q}}\\
& \leq c_5 b^{n(r-1) + n(r-\frac{1}{p})} (n + 2)^{\frac{1}{q}} \, b^{(n+3) (\frac{1}{p} - r)} + c_6 \, b^{2n (r-1)} (n + 2)^{\frac{1}{q}} \, b^{(n+3) (1 - r)}\\
& \leq c_7 \, b^{n(r-1)} \, n^{\frac{1}{q}}.
\end{align*}
Part \eqref{prp_3} gives us (using the fact that $r - 1 \leq 0$)
\begin{align*}
& \left( \sum_{j \in \N_0^2; \, j_1 \geq n} b^{(j_1 + j_2)(r - \frac{1}{p} + 1) q} \left( \sum_{m \in \D_j, \, \ell \in \B_j} | \mu_{jm\ell}|^p \right)^{\frac{q}{p}} \right)^{\frac{1}{q}}\\
& \leq c_8 \left( \sum_{j \in \N_0^2; \, j_1 \geq n} b^{(j_1 + j_2)(r - \frac{1}{p} + 1) q} \, b^{(-2j_1 - 2j_2)q} \, b^{(j_1 + j_2)\frac{q}{p}} \right)^{\frac{1}{q}}\\
& = c_8 \left( \sum_{\lambda = n}^\infty (\lambda + 1) b^{\lambda(r - 1) q} \right)^{\frac{1}{q}}\\
& \leq c_9 n^{\frac{1}{q}} b^{n(r-1)}
\end{align*}
and an analogous result for those $j \in \N_0^2$ with $j_2 \geq n$. From \eqref{prp_4} we conclude
\begin{align*}
& \left( \sum_{0 \leq j_1 < n; \, j_2 = -1} b^{(j_1 + j_2)(r - \frac{1}{p} + 1) q} \left( \sum_{m \in \D_j, \, \ell \in \B_j} | \mu_{jm\ell}|^p \right)^{\frac{q}{p}} \right)^{\frac{1}{q}}\\
& \leq c_{10} \left( \sum_{0 \leq j_1 < n; \, j_2 = -1} b^{(j_1 + j_2)(r - \frac{1}{p} + 1) q} \, b^{(j_1 + j_2) \frac{q}{p}} \, b^{(-n - j_1) q} \right)^{\frac{1}{q}}\\
& = c_{11} b^{-n} \left( \sum_{j_1 = 0}^{n-1} b^{j_1 q r} \right)^{\frac{1}{q}}\\
& \leq c_{11} b^{-n} b^{nr} = c_{11} b^{n(r-1)} \leq c_{11} b^{n(r-1)} n^{\frac{1}{q}}.
\end{align*}
Analogously one estimates the sum for those $j \in \N_{-1}^2$ with $j_1 = -1$ and $0 \leq j_2 < n$. From \eqref{prp_5} we have
\begin{align*}
& \left( \sum_{n \leq j_1; \, j_2 = -1} b^{(j_1 + j_2)(r - \frac{1}{p} + 1) q} \left( \sum_{m \in \D_j, \, \ell \in \B_j} | \mu_{jm\ell}|^p \right)^{\frac{q}{p}} \right)^{\frac{1}{q}}\\
& \leq c_{12} \left( \sum_{n \leq j_1; \, j_2 = -1} b^{(j_1 + j_2)(r - \frac{1}{p} + 1) q} \, b^{(j_1 + j_2) \frac{q}{p}} \, b^{-2j_1 q} \right)^{\frac{1}{q}}\\
& = c_{13} \left( \sum_{j_1 = n}^\infty b^{j_1 (r-1) q} \right)^{\frac{1}{q}}\\
& \leq c_{13} b^{n(r-1)} \leq c_{13} b^{n(r-1)} n^{\frac{1}{q}}
\end{align*}
again with analogous results for the sum with those $j \in \N_{-1}^2$ with $j_1 = -1$ and $n \leq j_2$. Finally, the last part gives us
\[ |\mu_{(-1,-1),(0,0),(1,1)}| \leq c_{14} b^{-n} \leq c_{14} b^{n(r-1)} n^{\frac{1}{q}}. \]
And the theorem is proved.

\end{proof}

\section{Final remarks}

The results from \cite[Chapter 6]{T10a} allow us to get additional results for Triebel-Lizorkin spaces with dominating mixed smoothness without any effort. First we define the spaces. We use the notation from the introduction. Let $0 < p,q \leq \infty$ and $r \in \R$. The Triebel-Lizorkin space with dominating mixed smoothness $S_{pq}^r F(\R^d)$ consists of all $f \in \mathcal{S}'(\R^d)$ with finite quasi-norm
\[ \left\| f | S_{pq}^r F(\R^d) \right\| = \left\| \left( \sum_{k \in \N_0^d} 2^{r (k_1 + \ldots + k_d) q} |\mathcal{F}^{-1}(\varphi_k \mathcal{F} f)(\cdot)|^q \right)^{\frac{1}{q}} | L_p(\R^d) \right\| \]
with the usual modification if $q = \infty$. The space $S_{pq}^r F(\Q^d)$ can be defined analogously to $S_{pq}^r B(\Q^d)$. In \cite{T10a} we find the following embeddings
\[ S_{p,\min(p,q)}^r B(\Q^d) \hookrightarrow S_{pq}^r F(\Q^d) \hookrightarrow S_{p,\max(p,q)}^r B(\Q^d) \]
and
\[ S_{p_1,q}^r F(\Q^d) \hookrightarrow S_{qq}^r B(\Q^d) \hookrightarrow S_{p_2,q}^r B(\Q^d) \]
for $0 < p_2 \leq q \leq p_1 < \infty$.
Using the main result of this note and these embeddings we get the following theorem

\begin{thm} \label{thm_hammersley_disc_f}

Let $1 \leq p,q \leq \infty$ and $0 \leq r < \frac{1}{\max(p,q)}$. Then for any integer $b\geq 2$ there are constants $c_1, c_2 > 0$ such that, for any $N \geq 2$, the discrepancy function of any point set $\mathcal{P}$ in $\Q^d$ with $N$ points satisfies
\[ \left\| D_{\mathcal{P}} | S_{pq}^r F(\Q^d) \right\| \geq c_1 \, N^{r-1} (\log N)^{\frac{d-1}{q}}, \]
and, for any $n \in \N$ and any generalized Hammersley type point set $\Rn$ with $a_n$ satisfying $|2a_n - n| \leq c_0$ for some constant $c_0 > 0$, we have
\[ \left\| D_{\Rn} | S_{pq}^r F(\Q^2) \right\| \leq c_2 \, b^{n(r-1)} \, n^{\frac{1}{q}}. \]

\end{thm}

The spaces $S_p^r H(\Q^d) := S_{p2}^r F(\Q^d)$ are called Sobolev spaces with dominating mixed smoothness. It is well known that $S_p^r H(\Q^d) = L_p(\Q^d)$. We can conclude the following.

\begin{thm} \label{thm_hammersley_disc_h}

Let $1 \leq p \leq \infty$ and $0 \leq r < \frac{1}{\max(p,2)}$. Then for any integer $b\geq 2$ there are constants $c_1, c_2 > 0$ such that, for any $N \geq 2$, the discrepancy function of any point set $\mathcal{P}$ in $\Q^d$ with $N$ points satisfies
\[ \left\| D_{\mathcal{P}} | S_p^r H(\Q^d) \right\| \geq c_1 \, N^{r-1} (\log N)^{\frac{d-1}{q}}, \]
and, for any $n \in \N$ and any generalized Hammersley type point set $\Rn$ with $a_n$ satisfying $|2a_n - n| \leq c_0$ for some constant $c_0 > 0$, we have
\[ \left\| D_{\Rn} | S_p^r H(\Q^2) \right\| \leq c_2 \, b^{n(r-1)} \, n^{\frac{1}{q}}. \]

\end{thm}

\end{document}